\documentclass{article}

\usepackage[utf8]{inputenc} % For \'o
\usepackage{amsmath}
\usepackage{amsthm}
\usepackage{amssymb}
\usepackage{cases}
\usepackage{mathrsfs}
\usepackage{ascmac}
\usepackage{framed}
\usepackage{fancybox}
\usepackage{graphicx}

\def\rnum#1{\expandafter{\romannumeral #1}}
\def\Rnum#1{\uppercase\expandafter{\romannumeral #1}} 

\theoremstyle{plain}
\newtheorem{thm}{Theorem}[section]

\newtheorem{lem}{Lemma}[section]
\newtheorem{rem}{Remark}[section]

\newcommand{\Rn}{\mathbb{R^{\textit{n}}}}
\newcommand{\N}{\mathbb{N}}
\newcommand{\Z}{\mathbb{Z}}
\newcommand{\R}{\mathbb{R}}
\newcommand{\C}{\mathbb{C}}

\newcommand{\B}{\dot{B}}

\newcommand{\tnormal}{\textnormal}
\newcommand{\dint}{\displaystyle \int}

\newcommand{\dis}{\displaystyle}
\newcommand{\vphi}{\varphi}
\newcommand{\vep}{\varepsilon}
\newcommand{\lan}{\langle}
\newcommand{\ran}{\rangle}

\newcommand{\X}{\B^{s(p)}_{p,\infty}}    % main norm 
\newcommand{\x}{\B^{s(p)+\tau_H}_{p,1}}     % positive smooth
\newcommand{\Y}{\B^s_{p,\infty}}        % support norm
\newcommand{\ZZ}{\B^{s(p)-\tau_L}_{p,\infty}}    % decay norm
\newcommand{\RL}{R_{-\Delta}}
\newcommand{\RA}{R_\mathcal{A}}
\begin{document}
\title{Asymptotic stability of the stationary Navier-Stokes flows in Besov spaces}

\author{Jayson Cunanan\thanks{Department of Mathematics, Graduate School of Science and Engineering, Saitama University, Saitama 338-8570, Japan E-mail: jcunanan@mail.saitama-u.ac.jp},
Takahiro Okabe\thanks{Department of Mathematics Education, Hirosaki University, Hirosaki  036-8560, Japan E-mail: okabe@hirosaki-u.ac.jp},
Yohei Tsutsui\thanks{Department of Mathematical Sciences, \ Shinshu University, \ Matsumoto, 390-8621, Japan E-mail: tsutsui@shinshu-u.ac.jp}}

\date{}

\maketitle{}

%\thanks{This work was partially supported by JSPS.}

%\begin{keyword}
%aaa
%\end{keyword}

\begin{abstract}
We discuss the asymptotic stability of stationary solutions to the incompressible Navier-Stokes equations on the whole space in Besov spaces with positive smoothness and low integrability.
A critical estimate for the semigroup generated by the Laplacian with a perturbation is main ingredient of the argument.
\end{abstract}

%%%%%%% Section 1 %%%%%%%%%%%%%%%%%%%%%%%%%%%%%%%%%
\section{Introduction}
In this article, we consider the asymptotic stability for stationary solutions of the incompressible Navier-Stokes equations:
\[
(N.S.)
\begin{cases}
\; \partial_t u - \Delta u + (u \cdot \nabla)u + \nabla \pi = f \\
\; \tnormal{div} \ u =0 \\
\; u(0) = a
\end{cases}
\]
in the whole space $\Rn$ with $n \ge 3$.
Here, $a$ and $f$ are given initial data and external force, and $u$ and $\pi$ stand for the velocity and pressure of the fluid, respectively.
The stationary solution $U$ of the Navier-Stokes equation is one independent of time;
\[
(S)
\begin{cases}
\; -\Delta U + (U \cdot \nabla)U + \nabla \Pi = f \\
\; \tnormal{div} \ U =0.
\end{cases}
\]
The purpose of this paper is to show the asymptotic behavior of the non-stationary solution $u$ when $a$ close to $U$.
More precisely, we prove that for such a $u$, $u(t) \to U$ as $t \to \infty$ in Besov spaces, provided that $f$ and $a-U$ are sufficient small.
If $u$ and $U$ are solutions to (N.S.) and (S), respectively, $w:= u-U$ and $b:= a-U$ satisfy 
\[
(E)
\begin{cases}
\; \partial_t w - \Delta w + (w \cdot \nabla)w + (w \cdot \nabla)U + (U \cdot \nabla)w + \nabla \psi = 0 \\
\; \tnormal{div} \ w =0 \\
\; w(0)=b
\end{cases}
\]
with $\psi := \pi - \Pi$.
Following the idea by Kozono-Yamazaki \cite{K-Y2}, we consider the semigroup generated by the Laplacian with a perturbation.
Using estimates established in this paper, we construct the solution $w$ to the integral equation of (E).
Decay properties for $w$ means the convergence of $u(t)$ to $U$ as $t \to \infty$.

\medskip

%%%% History of stationary sol.%%%%%%%%
The stationary solutions in unbounded domains were firstly treated in \cite{L}.
We refer \cite{B-S} and \cite{B-B-I-S} for further references.

%%%% Previous results for asymptotic stability %%%%%%
In the whole space $\Rn$, there many papers treating with the asymptotic stability problem for the incompressible Navier-Stokes equation.
Gallagher, Iftimie and Planchon \cite{G-I-P} proved the asymptotic stability of the trivial stationary solution $U \equiv 0$ in Besov space $\B^{s(p)}_{p,q}$ with $q < \infty$, where $s(p) := -1+n/p$, which is a scale invariant space with respect to the Navier-Stokes equations.
Auscher, Dubois and Tchamitchian \cite{A-D-T} showed a similar result in $BMO^{-1}$ for initial data from $VMO^{-1}$, which is the closure of $C^\infty_0$ in $BMO^{-1}$.
For the non-trivial case $U \not \equiv 0$, Kozono and Yamazaki \cite{K-Y2} showed the stability of stationary solution belonging to Morrey spaces in the topology of Besov-Morrey spaces.
Bjorland, Brandolese, Iftimie and Schonbek \cite{B-B-I-S} proved similar results in weak $L^p$ spaces with low integrability $p \in (n/2,n)$ in three dimensional case.
The result in \cite{K-Y2} corresponds to the case $p>n$.
In \cite{B-B-I-S}, they imposed a low integrability condition on the initial data and external force for treating with such small $p$.
Phan and Phuc considered the stability problem in function spaces based on capacities.
We aim to study the asymptotic stability in Besov spaces $\ZZ$ with some $\tau_L>0$ and $p \in (n/2,n)$, which restrict us the case $s(p) - \tau_L >0$, through product estimates Lemma \ref{product}.
For the aim, we assume better behavior of low frequencies on the data, similarly as \cite{B-B-I-S}.

\medskip

Our result concerning with the existence of the steady state solution can be read as follows.

%%%%%%%%% Theorem 1 %%%%%%%%%
\begin{thm} \label{sta}
Let $n \ge 3,\ n/2 < p < n$ and $s(p):= -1+n/p$.

(\rnum{1}): If $f$ is sufficiently small in $\B^{s(p) - 2}_{p,\infty}$ , then there uniquely exists $U \in \X$ solving (S) in the sense that 
\begin{equation} \label{U int}
U = (-\Delta)^{-1} \mathbb{P} f - (-\Delta)^{-1} \mathbb{P} \nabla (U \otimes U)\ \tnormal{in}\ \X,
\end{equation}
and having $\|U\|_{\X} \lesssim \|f\|_{\B^{s(p)-2}_{p,\infty}}$.

(\rnum{2}): Additionally, if we assume that $f \in \B^{s-2}_{p,\infty}$ with $s \in (0,1)$ and $\|f\|_{\X}$ is small, the solution above also satisfies $\|U\|_{\Y} \lesssim \|f\|_{\B^{s-2}_{p,\infty}}$.
\end{thm}

\begin{rem}
\begin{enumerate}
\item
Because all terms in (\ref{U int}) belong to $\X$, we can see that $U$ solves the differential equation
\[
-\Delta U + \mathbb{P} \nabla (U \otimes U) = \mathbb{P}f\ \tnormal{in}\ \B^{s(p)-2}_{p,\infty}
\]
and all terms in this equation belong to $\B^{s(p)-2}_{p,\infty}$.
For the definition of the projection $\mathbb{P}$ and $(-\Delta)^{\alpha/2}$ on homogeneous Besov spaces, see Subsection 2.2.

\item
We do not know if it is possible to construct stationary solutions in $\B^{s(p)-2}_{p,\infty}$ when $p \ge n$ or $p \le n/2$.
For later case, Bjorland-Brandolese-Iftimie-Schonbek \cite{B-B-I-S} gave a negative result in $L^p$ spaces.
\end{enumerate}
\end{rem}

\medskip

Now we are position to give our main result.
We treat with the stability of stationary solutions, constructed in Theorem \ref{sta}, in $\B^{s(p) \pm \tau}_{p,\infty}$ with some $\tau>0$ in the first and the second part, respectively.
%%%%%%%%%% Main Theorem  %%%%%%%%%%%%%%%%%%%%%
\begin{thm} \label{main}
Let $n \ge 3$ and $n/2 < p < n$, and suppose that $a \in L^{n,\infty}$ satisfies div $a = 0$ in $\mathcal{S}^\prime$.

(\rnum{1}):  Let $\tau_H \in (0,2-n/p)$.
If $f \in \B^{s(p)-2}_{p,\infty}$ is sufficiently small and $a-U$ is sufficiently small in $\X$, where $U \in \X$ is the corresponding stationary solution with $f$ constructed in (\rnum{1}) of Theorem \ref{sta}, then there uniquely exists $u \in BC \left((0,\infty); \X \right)$ so that $u - U \in BC \left((0,\infty),\X \cap \x \right)$ and solves the differential equation
\begin{equation} \label{u diffe}
\partial_t u(t) - \Delta u(t) + \mathbb{P} \nabla (u \otimes u)(t) = \mathbb{P}f\ \tnormal{in}\ \B^{s(p)-2}_{p,\infty}
\end{equation}
for $t>0$, with the initial condition $u(0)=a$ in the sense that for $\alpha \in[0,2]$
\[
\|u(t) - a\|_{\B^{s(p) - \alpha}_{p,\infty}} \lesssim t^{\alpha/2} \|a-U\|_{\X}.
\]
Moreover, it follows that $\|u(t)-U\|_{\X} \lesssim \|a-U\|_{\X}$ and
\[
\|u(t) - U\|_{\x} \lesssim t^{-\tau_H/2} \|a-U\|_{\X}.
\]
Further, $u(t) \to U$ in $\X$ as $t \to \infty$ if and only if $e^{-t\mathcal{A}}(a-U) \to 0$ in $\X$ as $t \to \infty$, see Section 3 for the definition of the semigroup $e^{-t\mathcal{A}}$.

(\rnum{2}): Let $s \in (0,s(p))$.
If we additionally assume that $f \in \B^{s-2}_{p,\infty},\ a-U \in \Y$, and $\|f\|_{\B^{s(p)-2}_{p,\infty}} + \|a-U\|_{\X}$ is sufficiently small, where $U \in \X \cap \Y$ is the corresponding stationary solution with $f$ constructed in (\rnum{2}) of Theorem \ref{sta}, then the solution $u$ above also satisfies $u \in BC\left((0,\infty); \Y \right)$ and for all $\tau_L \in (0,s(p)-s]$
\begin{equation} \label{u decay}
\|u(t) - U\|_{\ZZ} \lesssim t^{-\gamma/2} \|a-U\|_{\Y}\ \tnormal{where}\ \gamma := s(p) - \tau_L -s. 
\end{equation}
\end{thm}

%%% Remark
\begin{rem}
\begin{enumerate}
\item
Lemma \ref{emb} below tells us the following decay estimates;
\begin{align*}
\|u(t) - U\|_{L^{n/(1-\tau_H)}} & \lesssim \|u(t) - U\|_{\x} \lesssim t^{-\tau_H/2} \\
\|u(t)-U\|_{L^{n/(1 + \tau_L), \infty}} & \lesssim \|u(t) - U\|_{\ZZ} \lesssim t^{-\gamma/2},
\end{align*}
where
\[
p < \dfrac{np}{n-sp} \le \dfrac{n}{1+\tau_L} < n < \dfrac{n}{1-\tau_H} < \dfrac{np}{n-p}.
\]

\item
It is well-known that $|x|^{-1} \in \X \hookrightarrow L^{n,\infty}$ and $|x|^{-1} \in L^{n,q} \iff q=\infty$.
These mean that (\rnum{1}) does not covered by a result in Kozono-Yamazaki \cite{K-Y2}, though the statement is very similar as them and the strategy of the proof is the same as them.
In there the authors gave the stability of stationary solution in Besov-Morrey spaces.
When their space coincides with $\X$, their external force belongs to $\dot{W}^{-2,n}$.
The example above says that the class of our external forces, $\B^{s(p)-2}_{p,\infty}$, is not covered by $\dot{W}^{-2,n}$.

\item
Since $s(p) - \tau_L > 0$ in (\rnum{2}), the result is also independent of a result in Bjorland-Brandolese-Iftimie-Schonbek \cite{B-B-I-S}, in there the asymptotic stability was discussed in weak $L^p$ spaces.
\end{enumerate}
\end{rem}

\medskip

Applying the existence theorem of stationary solution in $L^{n,\infty}$ by Bjorland-Brandolese-Iftimie-Schonbek \cite{B-B-I-S}, we can show the similar result as follows.
Since the proof is almost same as that of Theorem \ref{main}, we omit it.
%%%%%%%% Theorem 1.3
\begin{thm} \label{main 2}
Let $n \ge 3$ and $n/2 < p < n$, and suppose that $a \in \L^{n,\infty}$ satisfies div $a = 0$ in $\mathcal{S}^\prime$.

(\rnum{1}): Let $\tau_H \in (0,2-n/p)$.
If $(-\Delta)^{-1} \mathbb{P} f \in L^{n,\infty}$ is sufficiently small and $a-U$ is sufficiently small in $\X$, where $U \in \L^{n,\infty}$ is the corresponding stationary solution with $f$ constructed in \cite{B-B-I-S}, then there uniquely exists $u \in BC \left((0,\infty); L^{n,\infty} \right)$ so that $u - U \in BC \left((0,\infty),\X \cap \x \right)$ and solves the differential equation
\begin{equation} \label{u diffe}
\partial_t u(t) - \Delta u(t) + \mathbb{P} \nabla (u \otimes u)(t) = \mathbb{P}f
\end{equation}
in $\dot{W}^{-2,(n,\infty)} := \{(-\Delta)^{-1}g ; g \in L^{n,\infty} \}$ for $t>0$, with the initial condition $u(0)=a$ in the sense that for $\alpha \in[0,2]$
\[
\|u(t) - a\|_{\B^{s(p) - \alpha}_{p,\infty}} \lesssim t^{\alpha/2} \|a-U\|_{\X}.
\]
Moreover, it follows that $\|u(t)-U\|_{\X} \lesssim \|a-U\|_{\X}$ and
\[
\|u(t) - U\|_{\x} \lesssim t^{-\tau_H/2} \|a-U\|_{\X}.
\]
Further, $u(t) \to U$ in $\X$ as $t \to \infty$ if and only if $e^{-t\mathcal{A}}(a-U) \to 0$ in $\X$ as $t \to \infty$, see Section 3 for the definition of the semigroup $e^{-t\mathcal{A}}$.

(\rnum{2}): Let $s \in (0,s(p))$ and $\ell(s) := np/(n-sp) \in (p,\infty)$.
If we additionally assume that $(-\Delta)^{-1} \mathbb{P} f \in L^{\ell(s),\infty},\ a-U \in \Y$, and $\|(-\Delta)^{-1} \mathbb{P} f\|_{L^{n,\infty}} + \|a-U\|_{\X}$ is sufficiently small, where $U \in \L^{n,\infty} \cap L^{\ell(s),\infty}$ is the corresponding stationary solution with $f$ constructed in \cite{B-B-I-S}, then the solution $u$ above also satisfies $u \in BC\left((0,\infty); L^{\ell(s),\infty} \right)$ and for all $\tau_L \in (0,s(p)-s]$
\begin{equation} \label{u decay}
\|u(t) - U\|_{\ZZ} \lesssim t^{-\gamma/2} \|a-U\|_{\Y}\ \tnormal{where}\ \gamma := s(p) - \tau_L -s. 
\end{equation}
\end{thm}

%%%% Reamrk
\begin{rem}
\begin{enumerate}
\item
For the existence of stationery solutions in $L^{n,\infty}$, used in above, see Theorem 2.2 in \cite{B-B-I-S}.

\item
From Lemma \ref{emb} below, we know that $\X \hookrightarrow L^{n,\infty}$ and $\Y \hookrightarrow L^{\ell(s),\infty}$.

\item
Theorems \ref{main} and \ref{main 2} are not covered by each other.
Although, the class of the external force in Theorem \ref{main} is larger than that of Theorem \ref{main 2}, the class of non-stationary solution in Theorem \ref{main} is smaller than that of Theorem \ref{main 2}.
\end{enumerate}
\end{rem}

\medskip

This paper is organized as follows.
In Section 2, we give the embedding between Besov spaces and weak $L^p$ spaces in Lemma \ref{emb}.
From this, we obtain a product estimate in Besov spaces.
This is a corollary of H$\ddot{\tnormal{o}}$lder inequality in weak $L^p$ spaces.
After that, following Bourdaud \cite{Bo}, and also \cite{B-C-D}, we define Fourier multipliers including the Helmholtz projection and the resolvent operator for the Laplacian on the homogeneous Besov spaces.
Section 3 is devoted to establish resolvent estimates for a perturbed Laplacian following Kozono-Yamazaki \cite{K-Y2}, and then give smoothing estimates for the semigrouop generated by the Laplacian.
A critical bound for Duhamel term is also proved.
In Section 4, after Theorem \ref{sta} is showed, we prove Theorem \ref{main}.

%%%%%%%%%%% Section 2 %%%%%%%%%%%%%%%%%%%%%%
\section{Preliminary}
Here, we recall definitions of function spaces, which we use, and collect inequalities, which are main tools for estimates in the following sections.

Throughout this paper we use the following notations.
$\mathcal{S}$ and $\mathcal{S}^\prime$ denote the Schwartz spaces of rapidly decreasing smooth functions and tempered distributions, respectively.
With reference \cite{Bo} and \cite{B-C-D}, we make use of a subspace $\mathcal{S}_h^\prime$ of $\mathcal{S}^\prime$ in the definition of the homogeneous Besov spaces.
This is a space of all tempered distributions $f$ fulfilling
\[
\psi(\lambda D)f := \mathcal{F}^{-1} \left[\psi(\lambda \cdot) \hat{f} \right] \to 0\ \tnormal{in}\ L^\infty\ \tnormal{as}\ \lambda \to \infty,
\]
for all $\psi \in \mathcal{S}$.
The symbols $\hat{}$ and $\mathcal{F}^{-1}$ denote the Fourier transform and its inverse.
$A \lesssim B$ means $A \le c B$ with a positive constant $c$.
$A \approx B$ means $A \lesssim B$ and $B \lesssim A$.
Let $s(p) := -1+n/p$.
For $p \in [1,\infty]$, we denote $p^\prime \in [1,\infty]$ satisfying $1/p+1/{p^\prime} = 1$.

%%%%%%%%%%%%%%%%%%%%%%%%%%%
\subsection{Function spaces}
Let us recall the definition of Besov spaces.
We fix $\vphi \in \mathcal{S}(\Rn)$ satisfying supp $\vphi \subset \{1/2 \le |\xi| \le 2 \}$ and $\dis \sum_{j \in \Z} \vphi \left(\dfrac{\xi}{2^j} \right) = 1$ for $\xi \in \Rn \backslash \{0\}$, and then $\vphi_j(D)f := \mathcal{F}^{-1} \left[\vphi \left(\dfrac{\cdot}{2^j} \right) \hat{f} \right] \in \mathcal{S}^\prime$ for $f \in \mathcal{S}^\prime$.
The fact that for any $f \in \mathcal{S}^\prime$, $\dis \sum_{j \ge 0} \vphi_j(D)f$ converges in the sense of $\mathcal{S}^\prime$ is well-known.
But, $\dis \sum_{j \le 0} \vphi_j(D)f$ may diverge in general.
We can see this when $f$ is a constant.
It is known that $\dis \sum_{j \le 0} \vphi_j(D)f$ converges in $\mathcal{S}^\prime$ and belongs to $\mathcal{S}^\prime_h$ if $f \in \mathcal{S}_h^\prime$.
Thus, it holds that $f = \dis \sum_{j \in \Z} \vphi_j(D)f \in \mathcal{S}^\prime_h$ converges in $\mathcal{S}^\prime$ for $f \in \mathcal{S}_h^\prime$.
We refer \cite{Bo} and \cite{B-C-D} for these facts.

%%%% Besov %%%%
For $p,q \in [1, \infty]$ and $s \in \R$, the homogeneous Besov space $\B^s_{p,q}$ is defined as
\[
\B^s_{p,q} := \left\{f \in \mathcal{S}_h^\prime; \|f\|_{\B^s_{p,q}} := \dis \left\| \left\{ 2^{js} \|\vphi_j(D)f \|_{L^p} \right\}_{j \in \Z} \right\|_{\ell^q} < \infty \right\}.
\]
These spaces become Banach spaces, if $s<n/p$ when $q>1$ or $s \le n/p$ when $q=1$, see \cite{Bo}, \cite{B-C-D} and \cite{D-M}.
$\B^s_{p,q}$ is invariant with respect to the scaling for the Navier-Stokes equations when $s=s(p)$.
We refer \cite{Bo}, \cite{B-C-D}, \cite{D-M} and \cite{S} for fundamental and important properties of Besov spaces.

\medskip

%%%%%%%%%%%%%%%%%%%%% Weak Lp
Next, we recall weak $L^p$ spaces.
Let $f$ be a measurable function and $p \in [1,\infty)$.
$f$ belongs to $L^{p,\infty}$ if
\[
\|f\|_{L^{p,\infty}} := \dis \sup_{\lambda > 0} \lambda \left| \left\{x \in \Rn; |f(x)| > \lambda \right\} \right|^{1/p} < \infty.
\] 
When $p \in (1,\infty)$, this space can be characterized by means of real interpolation;
\[
L^{p,\infty} = (L^{p_0}, L^{p_1})_{\theta, \infty}
\]
where $\theta \in (0,1),\ 1/p = (1-\theta)/{p_0} + \theta/{p_1}$ and $1< p_0 < p < p_1 \le \infty$.
The later space is a Banach space equipped the norm
\[
\dis \sup_{\lambda >0} \lambda^{-\theta} K(\lambda,f; L^{p_0}, L^{p_1}).
\]
The $K$-functional is defined by
\[
K(\lambda,f; L^{p_0}, L^{p_1}) := \inf_{\substack{f=f_0+f_1 \\ f_0 \in L^{p_0},\ f_1 \in L^{p_1}}} \left(\|f_0\|_{L^{p_0}} + \lambda \|f_1\|_{L^{p_1}} \right).
\]
$L^{n,\infty}$ is also scaling invariant space for the Naver-Stokes equations, and we know the following inclusions,
\[
\B^{n-1}_{1,\infty} \hookrightarrow \X \hookrightarrow L^{n,\infty} \hookrightarrow \B^{s(q)}_{q,\infty} \hookrightarrow BMO^{-1},
\]
where $1 \le p < n < q < \infty$, and the last space is treated by Koch-Tataru \cite{K-T}.
Kozono-Yamazaki \cite{K-Y1} constructed small global solutions in the frame work of Besov-Morrey spaces.

\medskip

We make use of next lemma to establish product estimates in Besov spaces in Lemma \ref{product} below.

%%%% Embedding %%%%%%%%%%%%%%%%%%
\begin{lem} \label{emb}
Let $1 < p < \infty,\ -n/{p^\prime} \le s<n/p$ and $\ell := np/(n-sp) \in [1,\infty)$.\\
(\rnum{1}): If $s>0$, then $\Y \hookrightarrow L^{\ell,\infty}$.

\noindent
(\rnum{2}): If $s<0$, then $L^{\ell,\infty} \hookrightarrow \Y$.
\end{lem}

\begin{proof}
Since $s<0$ implies $\ell < p$, (\rnum{2}) is immediately showed by Bernstein's inequality.

(\rnum{1}): We take $\ell_0,\ \ell_1 \ge 1$ so that $p < \ell_0 < \ell < \ell_1 < \infty$ and $2/\ell = 1/\ell_0 + 1/\ell_1$.
Bernstein's inequality yields that
\[
\left\|\sum_{j \ge \lambda_\ast} \vphi_j(D) f \right\|_{L^{\ell_0}} \lesssim 2^{\alpha \lambda_\ast} \|f\|_{\B^s_{p,\infty}}\ \tnormal{and}\ \left\|\sum_{j < \lambda_\ast} \vphi_j(D)f \right\|_{L^{\ell_1}} \lesssim 2^{\beta \lambda_\ast} \|f\|_{\B^s_{p,\infty}},
\]
where $\alpha = n(1/p - 1/\ell_0) - s < 0$ and $\beta = n(1/p - 1/\ell_1) - s > 0$.
Therefore, one has $K(\lambda, f; L^{\ell_0}, L^{\ell_1}) \lesssim \left(2^{\alpha \lambda_\ast} + \lambda 2^{\beta \lambda_\ast} \right) \|f\|_{\B^s_{p,\infty}}$.
Optimizing $\lambda_\ast \in \Z$, we obtain
\[
\|f\|_{(L^{\ell_0}, L^{\ell_1})_{1/2,\infty}} \le \dis \sup_{\lambda > 0} \lambda^{-1/2} \left(\left\|\sum_{j \ge \lambda_\ast} \vphi_j(D) f \right\|_{L^{\ell_0}} + \lambda \left\|\sum_{j < \lambda_\ast} \vphi_j(D)f \right\|_{L^{\ell_1} }\right) \lesssim \|f\|_{\B^s_{p,\infty}}.
\]
Thus, from the characterization of $L^{\ell, \infty}$ by real interpolation, the desired inequality $\|f\|_{L^{\ell, \infty}} \lesssim \|f\|_{\B^s_{p,\infty}}$ is proved.
\end{proof}

\medskip

%%%%%%%%%%  product %%%%%%%%%
\subsection{Product estimates}
The following estimate is applied to control the convection term $(g \cdot \nabla)h = \nabla \cdot (g \otimes h)$.
Instead of the paraproduct argument, we use embeddings between Besov spaces and weak $L^p$ spaces in Lemmas \ref{emb}.
As we see, the inequality is a consequence of the H$\ddot{\tnormal{o}}$lder inequality in weak $L^p$ spaces.
If we try to give the similar estimate by using paraproduct, it seems that $p \ge 2$ is needed.
This restriction causes a crucial problem for our purpose when $n=3$.

%%%%%%%%%%%% Product %%%%%%%%%
\begin{lem} \label{product}
Let $n \ge 3,\ n/2 < p <n $ and $0 < s < 1$.
Then,
\[
\|gh\|_{\B^{s-1}_{p,\infty}} \lesssim \|g\|_{L^{n,\infty}} \|h\|_{\Y}.
\]
\end{lem}

\begin{proof}
From (\rnum{2}) the previous lemma, we see that $L^{\ell, \infty} \hookrightarrow \B^{s-1}_{p,\infty}$ with $\ell = np/(n-(s-1)p) \in (1,\infty)$.
One has
\begin{align*}
\|gh\|_{\B^{s-1}_{p,\infty}} & \lesssim \|gh\|_{L^{\ell,\infty}} \\
& \le \|g\|_{L^{n,\infty}} \|h\|_{L^{r,\infty}} \\
& \lesssim \|g\|_{L^{n,\infty}} \|h\|_{\Y},
\end{align*}
where $r:=np/(n-sp) >1$ and (\rnum{1}) of Lemma \ref{emb} has been used in the third inequality.
\end{proof}

\medskip

%%%%%%% Operators %%%%%%%%
\subsection{Operators on homogeneous Besov spaces}
We make clear definitions of the projection $\mathbb{P}$ and the fractional derivative and integral operator $(-\Delta)^{\alpha/2}$ on homogeneous Besov spaces.
Moreover, we define $(\lambda + \Delta)^{-b/2}$ for $b > 0$ and $\lambda$ lying in a sector in $\C$.
Estimates for $(\lambda + \Delta)^{-1}$ are used in the next section.
Remark that our definition includes the end-point cases $p=1$ and $\infty$, although $\mathbb{P}$ is a singular integral operator.

Let $p,q \in [1,\infty],\ s \in \R$ and $f \in \B^s_{p,q}$.
For a function $m$ on $\Rn$, we consider the operator
\begin{equation} \label{m}
m(D)f := \dis \sum_{j \in \Z} \mathscr{F}^{-1} [m \vphi(\cdot/2^j) \hat{f}].
\end{equation}
Proposition 2.14 in \cite{B-C-D} is useful to show that the sum above converges in $\mathcal{S}^\prime$ and also $m(D)f \in \mathcal{S}_h^\prime$.
It is enough to prove the following:
\begin{align}
\dis \sup_{j \in \N} 2^{-j N} \left\|\mathscr{F}^{-1} \left[m \vphi (\cdot/2^j) \hat{f} \right] \right\|_{L^\infty} < \infty \quad \tnormal{with some}\ N \in \Z \label{L inf} \\
\dis \left\|\sum_{j=-m_2}^{-m_1} \mathscr{F}^{-1} \left[m \vphi (\cdot/2^j) \hat{f} \right] \right\|_{L^\infty} \to 0 \quad \tnormal{as}\ m_2 > m_1 \to \infty \label{conv}
\end{align}
Once we prove (\ref{L inf}) and (\ref{conv}) with any $\tilde{\vphi} \in \mathcal{S}$ with the Fourier support of $\tilde{\vphi}$ being a compact subset of $\Rn \backslash \{0\}$ instead of $\vphi$, we can see that $m(D)f$ is independent of $\vphi$, whenever $m \in C^\infty(\Rn \backslash \{0\})$.
Indeed, for such $\tilde{\vphi}$ and $\psi$,
\[
\left\lan \dis \sum_{j \in \Z} \mathcal{F}^{-1} \left[m \left(\vphi(\cdot/2^j) - \tilde{\vphi}(\cdot/2^j) \right) \hat{f} \right], \psi \right\ran = \left\lan \hat{f}, \sum_{j \in \Z} m \left[\vphi(\cdot/2^j) - \tilde{\vphi}(\cdot/2^j) \right] \hat{\psi} \right\ran = 0.
\]
Here, we remark that the last sum is finite one, and $m \hat{\psi} \in \mathcal{S}$.
Since the space consisting of such $\psi$ is dense in $L^2$, it turns out that $m(D)f$ is independent of the choice of cut-off function $\vphi$.

\medskip

%%% Helmhotz and fractional 
\subsubsection{Helmholtz projection and fractional Laplacian}
Let $a \in \R$, and assume that $m \in C^\infty(\Rn \backslash \{0\})$ is $a$-homogeneous, $m(\lambda \xi) = \lambda^a m(\xi)$ for $\lambda > 0$, and satisfies $|\partial^\alpha m(\xi)| \lesssim |\xi|^{a - |\alpha|}$ for all $\alpha \in (\N \cup \{0\})^n$ and $\xi \not = 0$.
We can see that $m(D)f \in \B^{s-a}_{p,q}$, provided that 
\begin{equation} \label{a}
a \ge s - n/p\ \tnormal{when}\ q=1\ \tnormal{or}\ a > s -n/p\ \tnormal{when}\ q \in (1,\infty].
\end{equation}
To see this, we observe that $\left\|\mathscr{F}^{-1} \left[m \vphi(\cdot/2^j) \right] \right\|_{L^1} = 2^{ja} \|\mathscr{F}^{-1}[m \vphi]\|_{L^1} \lesssim 2^{ja}$.
This and (\ref{a}) ensure that (\ref{L inf}) with $N=a+n/p-s$ and (\ref{conv}) hold.
Hence, we see from Lemma 2.23 in \cite{B-C-D} that $m(D)f \in \B^{s-a}_{p,q}$ and 
\[
\|m(D)f\|_{\B^{s-a}_{p,q}} \lesssim \left\| \left\{2^{j(s - a)} \left\|\mathscr{F}^{-1} \left[m \vphi (\cdot/2^j) \hat{f} \right] \right\|_{L^p} \right\}_{j \in \Z} \right\|_{\ell^q} \lesssim \|f\|_{\B^s_{p,q}}.
\]
Applying this argument, we can define the Helmholtz projection $\mathbb{P}f := m(D)f$ with $m(\xi) := (\delta_{i,j} + \xi_i \xi_j /|\xi|^2)_{1 \le i,j \le n}$ and obtain the boundedness
\[
\|\mathbb{P}f\|_{\B^s_{p,q}} \lesssim \|f\|_{\B^s_{p,q}}
\]
when
\begin{equation} \label{s}
s \le n/p\ \tnormal{when}\ q=1\ \tnormal{or}\ s<n/p\ \tnormal{when}\ q \in (1,\infty].
\end{equation}
We remark that this inequality holds even if $p=1, \infty$.
Similarly, let us denote $(-\Delta)^{\alpha / 2}f := m_a(D)f$ with $m_a(\xi) := |\xi|^a$.
Under the condition (\ref{a}), $(-\Delta)^{a/2}$ is an operator from $\Y$ to $\B^{s-a}_{p,q}$ with the estimates
\[
\|(-\Delta)^{a/2}f\|_{\B^{s-a}_{p,q}} \lesssim \|f\|_{\B^s_{p,q}}.
\]

\medskip

%%% Resolvent 
\subsubsection{Resolvent operator}
Let $b \ge 0,\ \omega \in (0,\pi/2)$ and
\[
S_\omega := \{z \in \C \backslash \{0\}; |\tnormal{arg}(z)| \ge \omega \}.
\]
For $\lambda \in S_\omega$, we denote $m_b(\xi) := (\lambda - |\xi|^2)^{-b/2}$ and observe that $|\partial^\alpha m_b(\xi)| \lesssim \min(|\lambda|^{-b/2}, |\xi|^{-b}) |\xi|^{-|\alpha|}$ for $\xi \not = 0$.
Thus it holds that $\|\mathscr{F}^{-1}[m_b(\cdot) \vphi(\cdot/2^j)]\|_{L^1} = \|\mathscr{F}^{-1}[m_b(2^j \cdot) \vphi(\cdot)]\|_{L^1} \lesssim \min(|\lambda|^{-b/2}, 2^{- jb})$.
Hence under the condition (\ref{s}), it turns out that (\ref{L inf}) with $N=n/p-s$ and (\ref{conv}) hold.
Therefore, we see that $(\lambda + \Delta)^{-b/2}f := m_b(D)f = \dis \sum_{j \in \Z} \mathscr{F}^{-1}[m_b \vphi(\cdot/2^j) \hat{f}]$ converges in $\mathcal{S}^\prime$ and belongs to $\mathcal{S}_h^\prime$.
Moreover, Lemma 2.23 in \cite{B-C-D} gives
\begin{align}
\|(\lambda + \Delta)^{-b/2}f\|_{\B^s_{p,q}} & \lesssim \left\| \left\{2^{js} \left\|\mathscr{F}^{-1} \left[m_b \vphi (\cdot/2^j) \hat{f} \right] \right\|_{L^p} \right\}_{j \in \Z} \right\|_{\ell^q} \lesssim |\lambda|^{-b/2} \|f\|_{\B^s_{p,q}}, \label{reso Laplacian 1} \\
\|(\lambda + \Delta)^{-b/2} f\|_{\B^{s+b}_{p,q}} & \lesssim \left\| \left\{2^{j(s+b)} \left\|\mathscr{F}^{-1} \left[m_b \vphi (\cdot/2^j) \hat{f} \right] \right\|_{L^p} \right\}_{j \in \Z} \right\|_{\ell^q} \lesssim \|f\|_{\B^s_{p,q}}. \label{reso Laplacian 2}
\end{align}
The implicit constants above are independent of $\lambda \in S_\omega$.
Especially, we define the resolvent operator $\RL(\lambda) := (\lambda + \Delta)^{-1}$.
If (\ref{s}) holds, then $(\lambda + \Delta)$ is an injection from $\Y \cap \B^{s+2}_{p,\infty} \subset \Y$ to $\Y$ and also $I=\RL(\lambda) (\lambda + \Delta) = (\lambda + \Delta) \RL(\lambda)$ on $\B^s_{p,q}$.

\medskip

Next, we shall consider $\B^s_{p,q} - \B^s_{p_0,q}$ estimates with $p_0 > p$ for $(\lambda + \Delta)^{-b/2}$ when
\begin{equation} \label{condition b}
b\ge n/p.
\end{equation}
To do that, we observe the estimate $\|\mathscr{F}^{-1}[m_b(\cdot) \vphi(\cdot/2^j)]\|_{L^r} \lesssim \min(|\lambda|^{-b/2}, 2^{- jb}) 2^{jn(1-1/r)}$ for $r \in (1,\infty)$ by interpolating the $L^1$-estimate above and the $L^\infty$-estimate,
\[
\|\mathscr{F}^{-1}[m_b \vphi(\cdot/2^j)]\|_{L^\infty} \lesssim \left\|m_b \vphi(\cdot/2^j) \right\|_{L^1} \lesssim \min(|\lambda|^{-b/2}, 2^{- jb}) 2^{jn}.
\]
Therefore, we obtain
\begin{equation} \label{reso Laplacian 3}
\|(\lambda + \Delta)^{-b/2}f\|_{\B^s_{p_0,q}} \lesssim \left\| \left\{2^{js} \left\|\mathscr{F}^{-1} \left[m_b \vphi (\cdot/2^j) \hat{f} \right] \right\|_{L^{p_0}} \right\}_{j \in \Z} \right\|_{\ell^q} \lesssim |\lambda|^{-(b-n(1/p - 1/{p_0}))/2} \|f\|_{\B^s_{p,q}}.
\end{equation}

\medskip

%%% Composition op
\subsubsection{Composition operator}
Let $0 \le a \le b$.
We consider the composition operator $m_{a,b}(D)f$ with $m_{a,b}(\xi) := |\xi|^a (\lambda - |\xi|^2)^{-b/2}$ where $\lambda \in S_\omega$.
It is not hard to see, from the previous sections, that if
\begin{equation} \label{ab}
a-b \ge s - n/p\ \tnormal{when}\ q=1\ \tnormal{or}\ a-b > s -n/p\ \tnormal{when}\ q \in (1,\infty],
\end{equation}
then $\dis \sum_{j \in \Z} \mathscr{F}^{-1}[m_{a,b} \vphi(\cdot/2^j) \hat{f}]$ converges in $\mathcal{S}^\prime,\ m_{a,b}(D) f \in \mathcal{S}_h^\prime$ and
\[
\|m_{a,b}(D)f\|_{\B^s_{p,q}} \lesssim |\lambda|^{-b/2} \|f\|_{\B^s_{p,q}}\ \tnormal{and}\ \|m_{a,b}(D)f\|_{\B^{s+b}_{p,q}} \lesssim \|f\|_{\B^s_{p,q}}.
\]
Further, we can see that $m_{a,b}(D) = m_a(D) \circ m_b(D) = m_b(D) \circ m_a(D)$.
In fact,
\[
m_a(D) \circ m_b(D)f = \dis \sum_{j \in \Z} \mathscr{F}^{-1}[m_a \vphi(\cdot/2^j) \widehat{m_b(D)f}] = \sum_{j \in \Z} \mathscr{F}^{-1}[m_a m_b \vphi(\cdot/2^j) \tilde{\vphi}(\cdot/2^j) \hat{f}] = m_{a,b}(D)f,
\]
where the cut-off function $\tilde{\vphi}$ is chosen so that $\vphi(\xi/2^j) \tilde{\vphi}(\xi/2^k) \equiv 0$ if $j \not = k$, and we have used the fact that the definitions of the operators are independent of test functions.
Similarly, we have $m_b(D) \circ m_a(D) = m_{a,b}(D)$.

\medskip

%%%%%%%%%%%%  section 3 Resolvent %%%%%%%%%%%%%%%%%%%%%%%%
\section{Resolvent estimates and a critical estimate for the semigroup}
Here, we discuss resolvent estimates for the Laplacian with the perturbation $\mathcal{B}$;
\begin{align*}
\mathcal{B}[w] & := \mathbb{P} \nabla \left(U \otimes w + w \otimes U \right) \\
\mathcal{A}[w] & := - \Delta w(t) + \mathcal{B}[w]
\end{align*}
where $U \in L^{n,\infty}$.
In this section, we assume that
\[
n \ge 3\ \tnormal{and}\ n/2 < p < n.
\]
This restriction steams from the presence of the perturbation $\mathcal{B}$ and Lemma \ref{product}.
We consider estimates for $\mathcal{A}$ with the domain $D(\mathcal{A}) : =\X \cap \B^{s(p)+2}_{p,\infty}$.
It is known that this is a Banach space equipped the norm $\|f\|_{\X} + \|f\|_{\B^{s(p)+2}_{p,\infty}}$, although the latter space is not so, see \cite{B-C-D}.
$D(\mathcal{A})$ is not a dense subspace of $\X$.

\subsection{Resolvent estimates for the perturbed operator $\mathcal{A}$}
We remark that $b=2$ fulfills the condition (\ref{condition b}), and then start with the estimates for $\mathcal{A}$ and $\mathcal{B}$.
%%%%%%%%% Estimates for A and B %%%%%%%%%%%%%%%%
\begin{lem} \label{AB}
It follows that for $s \in (0,1)$
\[
\|\mathcal{B}[w] \|_{\B^{s-2}_{p,\infty}} \lesssim \|U\|_{L^{n,\infty}} \|w\|_{\Y} \ \tnormal{and}\ \|\mathcal{A}[w]\|_{\B^{s-2}_{p,\infty}} \lesssim (1+ \|U\|_{L^{n,\infty}}) \|w\|_{\Y}.
\]
\end{lem}

\begin{proof}
This is done by the argument in Subsection 2.2. and Lemma \ref{product}.
\end{proof}

\medskip

Let $s \in (0,1)$.
Combining the estimate for $\mathcal{B}$ above with (\ref{reso Laplacian 2}), we see that if $\|U\|_{L^{n,\infty}}$ is sufficient small, then $(1 - \RL(\lambda) \mathcal{B})$ is invertible on $\Y$ with
\[
\left(1 - \RL(\lambda) \mathcal{B} \right)^{-1}  = \dis \sum_{\ell=0}^\infty \left[\RL(\lambda) \mathcal{B} \right]^\ell \ \tnormal{and}\ \left\|(1 - \RL(\lambda) \mathcal{B})^{-1} \right\|_{\Y \to \Y} \le 1.
\]
Because it holds that $(\lambda - \mathcal{A}) = (\lambda + \Delta) (I - \RL(\lambda) \mathcal{B})$ as a map from $D(A)$ to $\Y$, we see that
\[
\RA(\lambda) := \left(1 - \RL(\lambda) \mathcal{B} \right)^{-1} \RL(\lambda)
\]
is an operator from $\Y$ to itself, and satisfies $(\lambda - \mathcal{A}) \RA(\lambda) = \RA(\lambda) (\lambda - \mathcal{A}) =I$ on $\Y$ with the estimate
\[
\|\RA(\lambda)\|_{\Y \to \Y} \lesssim |\lambda|^{-1}.
\]
Moreover, we establish smoothing estimates.
To see these, we make use of an auxiliary operator
\[
\mathcal{C}_\theta:= (-\Delta)^{-(2-\theta)/2} \mathcal{B} (-\Delta)^{-\theta/2}
\]
with $\theta \in [0,2]$.
This operator has the following properties.
%%%%% Lemma for C
\begin{lem} \label{C}
Let $s<n/p$ and $\theta \in [0,2]$ with $0 < s + \theta < 1$.
Assume that $U$ is sufficiently small in $L^{n,\infty}$.

(\rnum{1}):
\[
\|C_\theta\|_{\Y \to \Y} \lesssim \|U\|_{L^{n,\infty}}.
\]

(\rnum{2}):
\[
\dis \sum_{j=0}^\infty \left[\RL(\lambda) \mathcal{B} \right]^j = (-\Delta)^{-\theta/2} \sum_{j=0}^\infty \left[\RL(\lambda) (-\Delta) \mathcal{C}_\theta \right]^j (-\Delta)^{\theta/2}.
\]

(\rnum{3}):
\[
\dis \left\|\sum_{j=0}^\infty \left[\RL(\lambda) (-\Delta) \mathcal{C}_\theta \right]^j \right\|_{\Y \to \Y} \lesssim \|U\|_{L^{n,\infty}}.
\]
\end{lem}

\begin{proof}
(\rnum{1}): This can be seen from the mapping properties of the negative powers of $(-\Delta)$ in Subsection 2.3 and Lemma \ref{AB}.

(\rnum{2}): Using the commutativity of $\RL(\lambda)$ and $(-\Delta)^{-\theta/2}$, we have
\[
\left[\RL(\lambda) \mathcal{B} \right]^2 = (-\Delta)^{-\theta/2} \left[\RL(\lambda) (-\Delta) \mathcal{C}_\theta \right]^2 (-\Delta)^{\theta/2}.
\]
Thus, the desired equality holds.

(\rnum{3}): Since
\[
\RL(\lambda) (-\Delta)g = \dis \sum_{j \in \Z} \mathcal{F}^{-1} [m(\xi) \vphi(\xi/2^j) \hat{g}]
\]
with $m(\xi) := (\lambda + |\xi|^2)^{-1} |\xi|^2$ and $|\partial^\alpha m(\xi)| \lesssim \min(|\lambda|^{-1}, |\xi|^{-2}) |\xi|^{2-|\alpha|} \le |\xi|^{-|\alpha|}$ for all $\xi \not = 0$, one has the boundedness $\|\RL(\lambda) (-\Delta)g\|_{\Y} \lesssim \|g\|_{\Y}$.
Combining this with (\rnum{1}) and the smallness of $\|U\|_{L^{n,\infty}}$, one has the absolutely convergence of the sum.
\end{proof}

%%%%%%%%%%%%% Resolvent estimate for A %%%%%%%%%%%%%%%%
\begin{lem} \label{resolvent}
Let $-2<s<1$ and $-2 \le \tau \le 2$ with $-2<s+\tau < 1$.
Assume that $\|U\|_{L^{n,\infty}}$ is sufficiently small.

(\rnum{1}):
\[
\|\RL(\lambda) \mathcal{B} \RA(\lambda)\|_{\Y \to \B^{s+\tau}_{p,\infty}} \lesssim |\lambda|^{-(2-\tau)/2} \|U\|_{L^{n,\infty}}.
\]

(\rnum{2}): If $\tau \ge 0$, then
\[
\|\RA(\lambda)\|_{\Y \to \B^{s+\tau}_{p,\infty}} \lesssim |\lambda|^{-(2-\tau)/2} (1+\|U\|_{L^{n,\infty}}).
\]
\end{lem}

\begin{proof}
(\rnum{1}):
The conditions on exponents ensure the existence of $\eta \in [0,1] \cap (-s/2,(1-s)/2) \cap (\tau/2,(\tau+2)/2)$.
Thus, $\tilde{\eta}:=(\tau +2)/2 - \eta \in [0,1]$ enjoys $\tau + 2 = 2(\eta + \tilde{\eta})$.
Then, we take $\theta \in [0,2]$ so that $s-n/p + 2\eta < \theta$.

Following Kozono-Yamazaki \cite{K-Y2}, we rewrite
\[
\RL(\lambda) \mathcal{B} \RA(\lambda) = \RL(\lambda) (-\Delta)^{(2-\theta)/2} \mathcal{C}_\theta \dis \sum_{j=0}^\infty \left[\RL(\lambda) (-\Delta) \mathcal{C}_\theta \right]^j (-\Delta)^{\theta/2} \RL(\lambda),
\]
and then we can obtain
\[
\begin{cases}
\; \left\|\RL(\lambda) (-\Delta)^{(2-\theta)/2} \right\|_{\B^{s+\tau+(2-\theta)-2 \tilde{\eta}}_{p,\infty} \to \B^{s+\tau}_{p,\infty}} \lesssim |\lambda|^{-(1-\tilde{\eta})} \\
\; \left\|\mathcal{C}_\theta \dis \sum_{j=0}^\infty \left[\RL(\lambda) (-\Delta) \mathcal{C}_\theta \right]^j \right\|_{\B^{s+\tau+(2-\theta)-2 \tilde{\eta}}_{p,\infty} \to \B^{s+\tau+(2-\theta)-2 \tilde{\eta}}_{p,\infty}} \lesssim \|U\|_{L^{n,\infty}} \\
\; \left\|(-\Delta)^{\theta/2} \RL(\lambda) \right\|_{\Y \to \B^{s-\theta+2\eta}_{p,\infty}} \lesssim |\lambda|^{-(1-\eta)}.
\end{cases}
\]
Here, we remark that $s+\tau+(2-\theta) - 2 \tilde{\eta} = s-\theta+2\eta$.
The first and last inequalities follows from the argument in Subsection 2.3.
The second one is combination of (\rnum{1}) and (\rnum{3}) in Lemma \ref{C}. 
Thus, the desired inequality is obtained.

\medskip

(\rnum{2}):
Interpolating (\ref{reso Laplacian 1}) and (\ref{reso Laplacian 2}), we have
\[
\|\RL(\lambda) f\|_{\B^{s+\tau}_{p,\infty}} \lesssim |\lambda|^{-(2-\tau)/2} \|f\|_{\Y}.
\]
Since
\begin{equation} \label{reso equ}
\RA(\lambda) = \RL(\lambda) + \RL(\lambda) \mathcal{B} \RA(\lambda),
\end{equation}
the estimate for $\RA(\lambda)$ is showed by this and (\rnum{1}).
\end{proof}

\medskip

Now we are position to define the semigroup with respect to $\mathcal{A}$ by the Dunford integral
\[
e^{-t\mathcal{A}} := \dfrac{1}{2\pi i} \dint_\Gamma e^{-t \lambda} \RA(\lambda) d\lambda\quad (t>0),
\]
where $\Gamma := \Gamma_- \cup \Gamma_0 \cup \Gamma_+$,
\[
\Gamma_0 := \{z \in \C; z=e^{i\psi},\ |\psi| \ge \theta \},\ \Gamma_{\pm} := \{z \in \C; z= re^{\pm i \theta}, r \ge 1 \},
\] 
and $\Gamma_\pm$ are oriented upwards and $\Gamma_0$ is counterclockwise, with some $\theta \in (\omega, \pi/2)$.
We refer \cite{Si} and \cite{L} for the standard properties of the semigroup.

\medskip

Next estimates for the semigroup is used in the proof of Theorem \ref{main}.
Similar estimates for the heat semigroup were showed by Kozono-Ogawa-Taniuchi \cite{K-O-T}.
%%%%%%%%%% Smoothing estimates for the semigroup %%%%%%%%%%%%%%%%
\begin{lem} \label{semigroup}
Let $-2<s<1$.
If $\|U\|_{L^{n,\infty}}$ is sufficiently small, then the the following estimates hold for $t>0$. \\
({\rnum{1}}):
\[
\|e^{-t \mathcal{A}}f\|_{\Y} \lesssim (1+\|U\|_{L^{n,\infty}}) \|f\|_{\Y}.
\]

\noindent
(\rnum{2}): For $\tau \in (0,1-s)$,
\[
\|e^{-t\mathcal{A}}f\|_{\B^{s+\tau}_{p,1}} \lesssim t^{-\tau/2} (1+\|U\|_{L^{n,\infty}}) \|f\|_{\Y}.
\]

\noindent
({\rnum{3}}): For $\tau \in [0,2]$ with $\tau \le 2+s$,
\[
\|e^{-t \mathcal{A}} f - f\|_{\B^{s-\tau}_{p,\infty}} \lesssim t^{\tau/2} \|f\|_{\Y}.
\]
\end{lem}

\begin{rem}
\begin{enumerate}
\item
Especially, (\rnum{1}) with large $\tau$ is important in Lemma \ref{critical} below.
In there, we use (\rnum{1}) with $\tau > 2$.

\item
Remark that we assume that $U \in L^{n,\infty}$ rather than $U \in L^n$.

\item
It seems that these estimates can not be deduced from similar results in Besov-Morrey spaces by Kozono-Yamazaki \cite{K-Y2}.
\end{enumerate}
\end{rem}

\begin{proof}
(\rnum{1}): This is an easy consequence of Lemma \ref{resolvent}.

(\rnum{2}): This inequality with $\B^{s+\tau}_{p,\infty}$ in the place of $\B^{s+\tau}_{p,1}$ immediately follows from the same reason as (\rnum{1}).
Applying an interpolation inequality by Machihara-Ozawa \cite{M-O}, we deduce the desired inequality as follows; if $0 < \tau_1 < \tau < \tau_2 < 1-s,\ \tau = (1-\theta) \tau_1 + \theta \tau_2$ with $\theta \in (0,1)$
\begin{align*}
\|e^{-t\mathcal{A}}f\|_{\B^{s + \tau}_{p,1}} & \lesssim \|e^{-t\mathcal{A}}f\|_{\B^{s + \tau_1}_{p,\infty}}^{1-\theta} \|e^{-t\mathcal{A}}f\|_{\B^{s + \tau_2}_{p,\infty}}^\theta \\
& \lesssim t^{-\tau/2} (1+\|U\|_{L^{n,\infty}}) \|f\|_{\Y}.
\end{align*}

(\rnum{3}): To show this, we use the identity (\ref{reso equ}) and thus
\begin{align*}
e^{-t \mathcal{A}}f-f = e^{t\Delta}f - f + \dfrac{1}{2\pi i} \dint_\Gamma e^{-t \lambda} \RL(\lambda) \mathcal{B} \RA(\lambda) f d\lambda. 
\end{align*}
Because $m(\xi) := e^{-|\xi|^2/4} -1$ fulfills $|\partial^\alpha m(\xi)| \lesssim |\xi|^{2-|\alpha|}$, the operator norm of the first term is controlled by $t^{\tau/2}$.
On the other hand, (\rnum{1}) in Lemma \ref{resolvent} yields
\[
\left\|\dint_\Gamma e^{-t \lambda} \RL(\lambda) \mathcal{B} \RA(\lambda) d\lambda \right\|_{\B^s_{p,\infty} \to \B^{s-\tau}_{p,\infty}} \lesssim t^{\tau/2} \|U\|_{L^{n,\infty}}.
\]
Thus, the proof is completed.
\end{proof}

\medskip

%%%%%%%%%%% Differentiability at the origin %%%%%%%%%%%%%%%%%
In next lemma, we see the differentiability of $e^{-t\mathcal{A}}f$ with respect to $t$.

\begin{lem} \label{reso last}
Let $0<s<1$ and $0 \le \tau \le 2$ with $\tau <s$.
Suppose that $\|U\|_{L^{n,\infty}}$ is sufficiently small.
Then,
\[
\left\|\dfrac{e^{-t \mathcal{A}}f-f}{t} + \mathcal{A}f \right\|_{\B^{s-2-\tau}_{p,\infty}} \lesssim t^{\tau/2} \|f\|_{\Y}.
\]
\end{lem}

\begin{proof}
Similarly as in \cite{K-Y2} and \cite{P-P}, we write
\begin{align*}
\dfrac{e^{-t \mathcal{A}}f-f}{t} + \mathcal{A}f & = \dfrac{1}{2 \pi i t} \dint_\Gamma e^{-t\lambda} \left[\RA(\lambda) - \lambda^{-1} \right] f d\lambda + \mathcal{A} f \\
& = \dfrac{1}{2 \pi i t} \dint_{\Gamma^\prime} \dfrac{e^{-\lambda}}{\lambda} \RA(\lambda/t) \mathcal{A}f d\lambda - \dfrac{1}{2\pi i} \dint_{\Gamma^\prime} \dfrac{e^{-\lambda}}{\lambda^2} \mathcal{A}f d\lambda \\
& = \dfrac{1}{2 \pi i} \dint_{\Gamma^\prime} \dfrac{e^{-\lambda}}{\lambda^2} \mathcal{A} \RA(\lambda/t) \mathcal{A} f d\lambda.
\end{align*}
where $\Gamma^\prime := \{t\lambda ; \lambda \in \Gamma \}$.
We have used the identities $\RA(\lambda) - \lambda^{-1} = \lambda^{-1} \RA(\lambda) \mathcal{A} = \lambda^{-1} \mathcal{A} \RA(\lambda)$ in the last two equalities.
Using Lemmas \ref{AB} and \ref{resolvent}, we conclude as follows;
\[
\left\| \dfrac{e^{-t \mathcal{A}}f-f}{t} + \mathcal{A}f \right\|_{\B^{s-2-\tau}_{p,\infty}} \lesssim t^{\tau/2} (1+\|U\|_{L^{n,\infty}})^2 \dint_{\Gamma^\prime} e^{-\lambda} \lambda^{-(2+\tau/2)} d\lambda \|f\|_{\Y}.
\]
\end{proof}

\medskip

%%%%%%%% critical %%%%%%%%%%%%
\subsection{A critical estimate}
We end this section with a critical estimate for the semigroup.
This type estimate was firstly proved by Meyer \cite{M} in $L^{n,\infty}$.
Yamazaki \cite{Y} independently proved it, on domains, applying estimates in the dual spaces.
Our estimate is an analogy of one in \cite{S-T-V} and \cite{T}, in there the estimate for heat semigroup was considered.

\begin{lem} \label{critical}
Let $0<s<1$.
It holds that
\[
\left\|\dint_{t_0}^t e^{-(t-\tau) \mathcal{A}} \mathbb{P} g(\tau) d\tau \right\|_{\B^s_{p,\infty}} \lesssim \dis \sup_{t_0 < \tau < t} \|g(\tau)\|_{\B^{s-2}_{p,\infty}}
\]
for $- \infty \le t_0 < t < \infty$.
\end{lem}

\begin{rem}
\begin{enumerate}
\item
Taking the norm in the integral, one has from Lemma \ref{semigroup}, $\|e^{- (t-\tau) \mathcal{A}} \mathbb{P} g(\tau)\|_{\B^s_{p,\infty}} \lesssim (t-\tau)^{-1} \|g(\tau)\|_{\B^{s-2}_{p,\infty}}$, which diverges the integral at $\tau=t$.
This difficulty can be overcome by the characterization $\B^s_{p,\infty} = (\B^{s_0}_{p,\infty}, \B^{s_1}_{p,\infty})_{\theta, \infty}$ with $s= (1-\theta)s_0 + \theta s_1$.

\item
This inequality is related to the $L^\infty$-maximal regularity.
Indeed, combining Lemma \ref{critical} eliminating $\mathbb{P}$ and Lemma \ref{AB}, we obtain the following.
\begin{align*}
\left\|\mathcal{A} \dint_0^t e^{-(t-\sigma) \mathcal{A}} f(\sigma) d\sigma \right\|_{\B^{s(p)-2}_{p,\infty}} & \lesssim \left\| \dint_0^t e^{-(t-\sigma) \mathcal{A}} f(\sigma) d\sigma \right\|_{\X} \\
& \lesssim \sup_{0 < \sigma < t} \|f(\sigma)\|_{\B^{s(p)-2}_{p,\infty}}.
\end{align*}
In other word, if $u(t) := \dint_0^t e^{-(t-\sigma) \mathcal{A}} f(\sigma) d\sigma$, then
\[
\dis \|\mathcal{A}u\|_{L^\infty \left((0,\infty); \B^{s(p)-2}_{p,\infty} \right)} \lesssim \|f\|_{L^\infty \left((0,\infty);\B^{s(p)-2}_{p,\infty} \right)}.
\]
Further, if $u$ solves the differential equation $\partial_tu + \mathcal{A}u = f$ with $u(0)=0$, then it follows that
\[
\|\partial_t u\|_{L^\infty \left((0,\infty); \B^{s(p)-2}_{p,\infty} \right)} + \|\mathcal{A}u\|_{L^\infty \left((0,\infty); \B^{s(p)-2}_{p,\infty} \right)} \lesssim \|f\|_{L^\infty \left((0,\infty); \B^{s(p)-2}_{p,\infty} \right)}.
\]

\end{enumerate}
\end{rem}

\medskip

\begin{proof}
Firstly, we rewrite the integral as follows;
\[
\dint_{t_0}^t e^{-(t-\tau) \mathcal{A}} \mathbb{P} g(\tau) d\tau = \dint_0^\infty e^{-\tau \mathcal{A}} \mathbb{P} \tilde{g}(\tau) d\tau
\]
with $\tilde{g}(\tau) := g(t-\tau) \chi_{(t_0,t)}(\tau)$.
Here, we take $\vep \in (0,1-s)$.
From the characterization $\B^s_{p,\infty} = (\B^{s-\vep}_{p,\infty}, \B^{s+\vep}_{p,\infty})_{1/2,\infty}$, we bound
\[
\left\|\dint_0^\infty e^{- \tau \mathcal{A}} \mathbb{P} \tilde{g}(\tau) d\tau \right\|_{\B^s_{p,\infty}} \lesssim \dis \sup_{\lambda >0} \lambda^{-1/2} \left(\|\Rnum{1}\|_{\B^{s-\vep}_{p,\infty}} + \lambda \|\Rnum{2}\|_{\B^{s+\vep}_{p,\infty}} \right),
\]
where $\Rnum{1} := \dint_0^{\lambda_\vep} e^{- \tau \mathcal{A}} \mathbb{P} \tilde{g}(\tau) d\tau$ and $\Rnum{2} := \dint_{\lambda_\vep}^\infty e^{- \tau \mathcal{A}} \mathbb{P} \tilde{g}(\tau) d\tau$ with $\lambda_\vep >0$.
Using the smoothing estimates for the semigroup and the bounedness of $\mathbb{P}$, one can see
\[
\|\Rnum{1}\|_{\B^{s-\vep}_{p,\infty}} \lesssim \lambda_\vep^{\vep/2} \dis \sup_{\tau > 0} \|\tilde{g}(\tau)\|_{\B^{s-2}_{p,\infty}}\ \ \tnormal{and}\ \ \|\Rnum{2}\|_{\B^{s+\vep}_{p,\infty}} \lesssim \lambda_\vep^{-\vep/2} \sup_{\tau >0} \|\tilde{g}(\tau)\|_{\B^{s-2}_{p,\infty}}.
\]
Optimizing $\lambda_\vep >0$, that is $\lambda_\vep = \lambda^{1/\vep}$, we obtain the desired inequality.
\end{proof}

\medskip

%%%%%  Section 4    Proof of theorems  %%%%%%%%%%%%%%%%
\section{Proof of Theorems \ref{sta} and \ref{main}}
We prove Theorem \ref{sta} in Subsection 4.1 and Theorem \ref{main} in Subsection 4.2, respectively.

\subsection{Construction of stationary solutions}
We construct stationary solutions $f$ by using successive approximations;
\[
U_0:= (-\Delta)^{-1} \mathbb{P} f\quad \tnormal{and}\quad U_{m+1} := U_0 - (-\Delta)^{-1} \mathbb{P}  \nabla (U_m \otimes U_m),\ (m = 0, 1, 2, \cdots).
\]
The boundedness of the projection $\mathbb{P}$ on $\B^{s(p)-2}_{p,\infty}$ and Lemma \ref{product} give us
\begin{align*}
\|U_{m+1}\|_{\X} & \lesssim \|f\|_{\B^{s(p)-2}_{p,\infty}} + \|U_m \otimes U_m\|_{\B^{s(p) - 1}_{p, \infty}} \\
& \lesssim \|f\|_{\B^{s(p)-2}_{p,\infty}} + \|U_m\|_{\X}^2,
\end{align*}
which is enough to complete the proof of (\rnum{1}).

Since one has
\begin{align*}
\|U_{m+1}\|_{\Y} & \lesssim \|f\|_{\B^{s-2}_{p,\infty}} + \|U_m \otimes U_m\|_{\B^{s-1}_{p,\infty}} \\
& \lesssim \|f\|_{\B^{s-2}_{p,\infty}} + \|U_m\|_{\X} \|U_m\|_{\Y},
\end{align*}
the proof of (\rnum{2}) is completed.

\medskip

%%%%%%  main part %%%%%%%%%%%%%%%
\subsection{Construction of non-stationary solutions and their asymptotic stability}
Let us define $\vep := \|a-U\|_{\X}$.
We consider the equation
\[
(E)
\begin{cases}
\; \partial_t w + \mathcal{A} [w] + \mathbb{P} (w \cdot \nabla)w = 0 \\
\; w(0)=b:= a-U.
\end{cases}
\]
For simplicity, we denote
\[
\|w\|_{X(t)} := \dis \sup_{\sigma \le t} \|w(\sigma)\|_{\X}\ \tnormal{and}\ \|w\|_{X_\tau(t)} := \sup_{\sigma \le t} \sigma^{\tau/2} \|w(\sigma)\|_{\B^{s(p)+\tau}_{p,1}}.
\]
Define
\[
w_0(t) := e^{- t \mathcal{A}} b \quad \tnormal{and} \quad w_{m+1}(t) := w_0(t) - B(w_m,w_m)(t),
\]
where
\[
B(g,h)(t) := \dint_0^t e^{-(t - \sigma) \mathcal{A}} \mathbb{P} \nabla (g \otimes h)(\sigma) d\sigma.
\]
Lemmas \ref{semigroup} and \ref{critical} give us the following:
\begin{align*}
\|w_0\|_{X(t)} & \lesssim \vep \\
\|w_0\|_{X_{\tau_H}(t)} & \lesssim \vep \\
\|B(g,h)\|_{X(t)} & \lesssim \|g\|_{X(t)} \|h\|_{X(t)} \\
\|B(g,h)\|_{X_{\tau_H}(t)} & \lesssim \left(\|g\|_{X(t)} + \|g\|_{X_{\tau_H}(t)} \right) \left(\|h\|_{X(t)} + \|h\|_{X_{\tau_H}(t)} \right).
\end{align*}
These imply that $\{w_m\}_{m=0}^\infty$ is a Cauchy sequence in both $\X$ and $\x$.
Let us give a proof of the last inequality.
Fix $\tau_0 \in (\tau_H/2,1-n/(2p))$.
Using Lemma \ref{semigroup} and the product estimate $\|gh\|_{\B^{s(p)+2\tau_0-1}_{p,\infty}} \lesssim \|g\|_{\B^{s(p)+\tau_0}_{p,1}} \|h\|_{\B^{s(p)+\tau_0}_{p,1}}$, we have
\[
\|B(g,h)(t)\|_{\x} \lesssim t^{-\tau/2} \|g\|_{X_{\tau_0}(t)} \|h\|_{X_{\tau_0}(t)}.
\]
Interpolation inequality $\|f\|_{X_{\tau_0}(t)} \lesssim \|f\|_{X(t)}^{1-\tau_0/\tau_H} \|f\|_{X_{\tau_H}(t)}^{\tau_0/\tau_H}$ and Young's inequality yield the desired one.
Because $w_0 \in BC \left((0,\infty), \X \cap \x \right)$ is showed from Lemma \ref{semigroup}, and
\begin{align*}
B(w_0,w_0)(t) - B(w_0,w_0)(t \pm \vep) & = \dint_0^{\min(t,t \pm \vep)} e^{- \sigma \mathcal{A}} \mathbb{P} \nabla \left[(w_0 \otimes w_0)(t - \sigma) - (w_0 \otimes w_0)(t \pm \vep - \sigma) \right] d\sigma \\
& + \dint_{\min(t,t \pm \vep)}^{\max(t, t \pm \vep)} e^{-\sigma \mathcal{A}} \mathbb{P} \nabla (w_0 \otimes w_0)(\max(t \pm \vep) - \sigma) d\sigma,
\end{align*}
we see $w_1 \in BC \left((0,\infty),\X \cap \x \right)$ from Lebesgue's convergence theorem.
Repeating this argument gives us $w_m \in BC \left((0,\infty),\X \cap \x \right)$.
Therefore, there exists $w \in BC \left((0,\infty),\X \cap \x \right)$ solving the integral equation
\[
w(t) = e^{-t\mathcal{A}}b - B(w,w)(t).
\]

\medskip

Next, we shall show that $w$, in fact, fulfills the differential equation in (E) in $\B^{s(p)-2}_{p,\infty}$.
We take $t>0$, and then $t_1,t_2 >0$ such that $t_1 < t < t_2$.
\begin{align*}
\dfrac{w(t_2) - w(t_1)}{t_2-t_1} = \dfrac{e^{-(t_2 - t_1)\mathcal{A}} - 1}{t_2-t_1} w(t) - \mathbb{P} \nabla (w \otimes w)(t) - (\Rnum{1} + \Rnum{2} + \Rnum{3}),
\end{align*}
where
\begin{align*}
\Rnum{1} & := \dfrac{e^{-(t_2-t_1)\mathcal{A}}-1}{t_2-t_1} (w(t) - w(t_1)) \\
\Rnum{2} & := \dfrac{1}{t_2-t_1} \dint_{t_1}^{t_2} e^{-(t_2-\sigma) \mathcal{A}} \mathbb{P} \nabla \left((w \otimes w)(\sigma) - (w \otimes w)(t) \right) d\sigma \\
\Rnum{3} & := \dfrac{1}{t_2 - t_1} \dint_{t_1}^{t_2} \left(e^{-(t_2-\sigma)\mathcal{A}} -1 \right) \mathbb{P} \nabla (w \otimes w)(t) d\sigma.
\end{align*}
It follows from Lemma \ref{reso last} that 
\[
\left\|\dfrac{e^{-(t_2-t_1) \mathcal{A}}-1}{t_2-t_1}w(t) + \mathcal{A}w(t) \right\|_{\B^{s(p)-2}_{p,\infty}} \lesssim (t_2-t_1)^{\tau/2} \|w(t)\|_{\x} \to 0\ \tnormal{as}\ t_2,t_1 \to t.
\]
Lemma \ref{semigroup} and the continuity of $w$ ensure the following convergences;
\[
\|\Rnum{1}\|_{\B^{s(p)-2}_{p,\infty}} \lesssim \|w(t) - w(t_1)\|_{\X} \to 0
\]
and 
\[
\|\Rnum{2}\|_{\B^{s(p)-2}_{p,\infty}} \lesssim \|w\|_{X(t_2)} \dis \sup_{t_1 \le \sigma \le t_2} \|w(\sigma) - w(t)\|_{\X} \to 0
\]
as $t_2,t_1 \to t$.
Since
\begin{align*}
\|\Rnum{3}\|_{\B^{s(p)-2}_{p,\infty}} & \lesssim (t_2-t_1)^{\tau_H/2} \|w \otimes w(t)\|_{\B^{s(p)+\tau_H -1}_{p,\infty}} \\
& \lesssim (t_2 - t_1)^{\tau_H/2} \|w(t)\|_{\X} \|w(t)\|_{\x} \to 0
\end{align*}
as $t_2, t_1 \to t$, we obtain that
\[
\partial_tw(t) + \mathcal{A}w(t) + \mathbb{P} \nabla (w \otimes w)(t) = 0\ \tnormal{in}\ \B^{s(p)-2}_{p,\infty}.
\]
Now, let us define $u := w + U \in BC \left((0,\infty), \X \right)$.
Thus, $u$ satisfies the differential equation
\[
\partial_tu(t) - \Delta u(t) + \mathbb{P} \nabla (u \otimes u)(t) = \mathbb{P} f\  \tnormal{in}\ \B^{s(p)-2}_{p,\infty},
\]
and also fulfills $\|u(t) - U\|_{\x} \lesssim \vep t^{-\tau_H/2} \to 0$ as $t \to \infty$.

\medskip

To end the proof of this part, we shall show the equivalence of convergences.
We borrow an argument from Cannone and Karch \cite{C-K}.
Suppose that $u(t) \to U$ in $\X$ as $t \to \infty$.
It is sufficient to verify that $B(w,w)(t)$ converges to zero.
Dividing the integral into two parts, we obtain that for $\delta \in (0,1)$,
\begin{equation} \label{convergence}
\|B(w,w)(t)\|_{\X} \lesssim \dint_0^{\delta t} (1-\sigma)^{-1} \|w(t\sigma)\|_{\X}^2 d\sigma + \dis \sup_{\delta t \le \sigma \le t} \|w(\sigma)\|_{\X}^2 \to 0
\end{equation}
as $t \to \infty$.
Next, we assume that $\dis \lim_{t \to \infty} \|e^{-t\mathcal{A}}b\|_{\X}=0$. 
It is enough to see that $B(w,w)$ tends to zero, again.
Applying the inequality in (\ref{convergence}), one can see that
\[
\dis \limsup_{t \to \infty} \|B(w,w)(t)\|_{\X} \lesssim \vep \left(1+ \log \dfrac{1}{1-\delta} \right) \limsup_{t \to \infty} \|w(t)\|_{\X}.
\]
As a consequence,
\[
\limsup_{t \to \infty} \|w(t)\|_{\X} \lesssim \vep \left(1+ \log \dfrac{1}{1-\delta} \right) \limsup_{t \to \infty} \|w(t)\|_{\X}.
\]
Taking sufficiently small $\delta$, we conclude that $u(t) \to U$ in $\X$ as $t$ tends to $\infty$.

\medskip

(\rnum{2}) Let us denote $M := \|a-U\|_{\Y}$ and
\[
\|w\|_{Y(t)} := \dis \sup_{\sigma \le t} \|w(\sigma)\|_{\Y}.
\]
Similarly as the estimate on $\X$, it follows from Lemmas \ref{product}, \ref{semigroup} and \ref{critical} that
\[
\|w_0\|_{Y(t)} \lesssim M\ \tnormal{and}\ \|B(g,h)\|_{Y(t)} \lesssim \min \left(\|g\|_{X(t)} \|h\|_{Y(t)}, \|g\|_{Y(t)} \|h\|_{X(t)} \right).
\]
Further, the continuity $w_m \in BC \left((0,\infty), \Y \right)$ is also easily seen.
Thus, we find that $w \in BC \left((0,\infty), \Y \right)$ with $\|w\|_{Y(t)} \lesssim M$.

We claim that
\begin{align} \label{main est}
\|B(g,h)(t)\|_{\ZZ} & \lesssim t^{-\gamma/2} \dis \sup_{0 < \tau \le t/2} \min \left(\|g(\tau)\|_{\X} \|h(\tau)\|_{\Y}, \|g(\tau)\|_{\Y} \|h(\tau)\|_{\X} \right) \nonumber \\
& + \sup_{t/2 \le \tau \le t} \min \left(\|g(\tau)\|_{\X} \|h(\tau)\|_{\ZZ}, \|g(\tau)\|_{\ZZ} \|h(\tau)\|_{\X} \right).
\end{align}
We show this inequality after completing the proof of Theorem \ref{main}.

We shall complete the proof, assuming (\ref{main est}).
Following  Bjorland, Brandolese, Iftimie and Schonbek \cite{B-B-I-S}, we have from (\ref{main est})
\begin{align*}
W_{m+1}(t) & := \dis \sup_{T \ge t} \|w_{m+1}(T)\|_{\ZZ} \\
& \lesssim \dis \sup_{T \ge t} \left(T^{-\gamma/2} M + \vep M T^{-\gamma/2} + \vep \sup_{T/2 \le \sigma \le T} \|w_m(\sigma)\|_{\ZZ} \right) \\
& \lesssim t^{-\gamma/2} M + \vep \sup_{t/2 \le \sigma} \|w_m(\sigma)\|_{\ZZ}.
\end{align*}
Hence, it turns out that $W_{m+1}(t) \lesssim t^{-\gamma/2} M + \vep W_m(t/2)$ for all $t \in (0,\infty)$, and then one has that $\dis \limsup_{m \to \infty} W_m(t) \lesssim t^{-\gamma/2} M$ for all $t \in (0,\infty)$.
For any $\theta \in \mathcal{S}$,
\[
\left|\lan w(t), \theta \ran \right| \le \dis \limsup_{m \to \infty} W_m(t) \|\theta\|_{\B^{-(s(p) - \tau_L)}_{p^\prime,1}}.
\]
By duality, we see that $\|w(t)\|_{\ZZ} \lesssim t^{-\gamma/2} M$.

We show the inequality (\ref{main est}).
Following the argument in \cite{B-B-I-S}, we decompose $B(g,h)(t) = e^{-t\mathcal{A}/2} B(g,h)(t/2) + B^\prime(g,h)(t)$, where
\[
B^\prime(g,h)(t) := \dint_{t/2}^t e^{-(t-\sigma) \mathcal{A}} \mathbb{P} \nabla (g \otimes h)(\sigma) d\sigma.
\]
The first part is controlled by Lemmas \ref{semigroup} and \ref{critical};
\begin{align*}
\|e^{-t \mathcal{A} /2} B(g,h)(t/2)\|_{\ZZ} & \lesssim t^{-\gamma/2} \|B(g,h)(t/2)\|_{\Y} \\
& \lesssim t^{-\gamma/2} \sup_{0 < \sigma \le t/2} \min \left(\|g(\sigma)\|_{\X} \|h(\sigma)\|_{\Y}, \|g(\sigma)\|_{\Y} \|h(\sigma)\|_{\X} \right).
\end{align*}
For the latter part, Lemma \ref{critical} also gives
\begin{align*}
\|B^\prime(g,h)(t)\|_{\ZZ} & \lesssim \dis \sup_{t/2 \le \sigma \le t} \|(g \otimes h)(\sigma)\|_{\B^{s(p) - \tau_L -1}_{p,\infty}} \\
& \lesssim \sup_{t/2 \le \sigma} \min \left(\|g(\sigma)\|_{\X} \|h(\sigma)\|_{\ZZ}, \|g(\sigma)\|_{\ZZ} \|h(\sigma)\|_{\X} \right).
\end{align*}
Applying Lemma \ref{product}, we complete the proof of the claim (\ref{main est}).

%%%%%%%%%%%%%%
\subsection*{Acknowledgements}
The third author would like to thank Professor Toshiaki Hishida for several comments.
The works of the first and third authors were partially supported by JSPS, through “Program to Disseminate Tenure Tracking System”.
The work of the second author was partially supported by JSPS, through Gand-in-Aid for Young Scientists (B) 17K14215.
The work of the third author was partially supported by JSPS, through Grand-in-Aid for Young Scientists (B) 15K20919.


\begin{thebibliography}{99}
\bibitem{A-D-T}
P. Auscher, S. Dubois and P. Tchamitchian,
\textit{On the stability of global solutions to Navier-Stokes equations in the space},
J. Math. Pures Appl. (9) {\bf 83} (2004), no. 6, 673–697.



\bibitem{B-C-D}
H. Bahouri, J.-Y. Chemin and R. Danchin,
\textit{Fourier analysis and nonlinear partial differential equations},
Grundlehren der Mathematischen Wissenschaften, vol. 343, Springer, Heidelberg, 2011.



\bibitem{B-B-I-S}
C. Bjorland, L. Brandolese, D. Iftimie and M.E. Schonbek,
\textit{$L^p$ solutions of the steady-state Navier-Stokes equations with rough external forces},
Comm. Partial Differential Equations {\bf 36} (2011), no. 2, 216–246. 


\bibitem{B-S}
C. Bjorland and M.E. Schonbek,
\textit{Existence and stability of steady-state solutions with finite energy for the Navier-Stokes equations in the whole space},
Nonlinearity {\bf 22} (2009) no. 7, 1615-1637.




\bibitem{Bo}
G. Bourdaud,
\textit{Realizations of homogeneous Besov and Lizorkin-Triebel spaces},
Math. Nachr. {\bf 286} (2013), no. 5-6, 476–491.



\bibitem{C-K}
M. Cannone and G. Karch,
\textit{About the regularized Navier-Stokes equations},
J. Math. Fluid Mech. {\bf 7} (1) (2005), 1-28.



\bibitem{D-M}
R. Danchin and P.B. Mucha,
\textit{Critical functional framework and maximal regularity in action on systems of incompressible flows},
Mém. Soc. Math. Fr. (N.S.) No. 143 (2015).


\bibitem{G-I-P}
I. Gallagher, D. Iftimie and F. Planchon,
\textit{Asymptotics and stability for global solutions to the Navier-Stokes equations},
Ann. Inst. Fourier (Grenoble) {\bf 53} (2003), no. 5, 1387–1424.



\bibitem{L}
A. Lunardi,
\textit{Analytic semigroups and optimal regularity in parabolic problems},
Birkhäuser, Basel (1995).



\bibitem{M-O}
S. Machihara and T. Ozawa,
\textit{Interpolation inequalities in Besov spaces},
Proc. Amer. Math. Soc. {\bf 131}, 1553-1556.



\bibitem{M}
Y. Meyer,
\textit{Wavelet, Paraproduct and Navier-Stokes Equations},
Current developments in mathematics. International Press, (1996), 105-212.




\bibitem{K-T}
H. Koch and D. Tataru,
\textit{Well-posedness for the Navier-Stokes equations},
Adv. Math. {\bf 157} (2001), 22-35.


\bibitem{K-O-T}
H. Kozono, T. Ogawa and Y. Taniuchi,
\textit{Navier-Stokes equations in the Besov space near $L^\infty$ and $BMO$}
Kyushu J. Math., {\bf 57} (2003), 303-324.


\bibitem{K-Y1}
H. Kozono and M. Yamazaki,
\textit{Semilinear heat equations and the Navier-Stokes equation with distributions in new function spaces as initial data},
Comm. Partial Differential Equations {\bf 19} (1994), no. 5-6, 959–1014. 


\bibitem{K-Y2}
H. Kozono and M. Yamazaki,
\textit{The stability of small stationary solutions in Morrey spaces of the Navier-Stokes equation},
Indiana Univ. Math. J. {\bf 44} (1995), no. 4, 1307–1336.


\bibitem{L}
J. Leray,
\textit{Étude de diverses équations intégrals non linéaires et de quelques problèmes que pose l'Hydrodynamique},
J. Math. Pures Appl. {\bf 12} (1933) 1–82.


\bibitem{P-P}
T.V. Phan and N.C. Phuc,
\textit{Stationary Navier-Stokes equations with critically singular external forces: existence and stability results},
Adv. Math. {\bf 241} (2013), 137–161. 


\bibitem{S}
Y. Sawano,
\textit{Homogeneous Besov spaces},
arXiv:1603.07889 [math.FA].


\bibitem{Si}
E. Sinestrari,
\textit{On the abstract Cauchy problem of parabolic type inn spaces of continuous functions},
J. Math. Anal. Appl. {\bf 107} (1985), 16-66.


\bibitem{S-T-V}
Y. Sugiyama, Y. Tsutsui and J.L.L. Vel\'azquez,
\textit{Global solutions to a chemotaxis system with non-diffusive memory},
J. Math. Anal. Appl. {\bf 410} (2014), no. 2, 908-917.


\bibitem{T}
Y. Tsutsui,
\textit{Bounded global solutions to a Keller–Segel system with nondiffusive chemical in $\R^n$},
J. Evol. Equ. {\bf 17} (2017), no. 2, 627–640. 


\bibitem{Y}
M. Yamazaki,
\textit{The Navier-Stokes equations in the weak-$L^n$ space with time-dependent external force},
Math. Ann. {\bf 317} (2000), no. 4, 635–675.
\end{thebibliography}
\end{document}